\documentclass[envcountsect,referee]{svjour3}

\smartqed  
\usepackage{graphicx}
\usepackage{latexsym,bm,amsmath,amssymb,mathrsfs}
\usepackage{rotating}
\usepackage{multirow,booktabs,cases}
\usepackage[misc]{ifsym}
\usepackage[style=1]{mdframed}
\usepackage{algorithm,algorithmic}
\usepackage[colorlinks=true]{hyperref}
\hypersetup{urlcolor=blue,citecolor=blue,linkcolor=blue}
\usepackage{epstopdf}

\usepackage[justification=centering]{caption}

\def\[{\begin{equation}}
\def\]{\end{equation}}

\numberwithin{equation}{section}
\allowdisplaybreaks

\begin{document}
\graphicspath{{./PIC/}}

\title{A nonnegativity preserving algorithm for multilinear systems with nonsingular ${\mathcal M}$-tensors}

\titlerunning{A nonnegativity preserving algorithm for multilinear systems}     

\author{Xueli Bai \and Hongjin He \and Chen Ling \and Guanglu Zhou }


\institute{X. Bai \at
School of Mathematics, Tianjin University, Tianjin, 300350, China.\\
\email{baixueli@tju.edu.cn}
\and H. He (\Letter) \and C. Ling \at
Department of Mathematics, Hangzhou Dianzi University, Hangzhou, 310018, China.\\
\email{hehjmath@hdu.edu.cn}
\and C. Ling \at
\email{macling@hdu.edu.cn}
\and G. Zhou \at
Department of Mathematics and Statistics, Curtin University, Perth, WA, Australia.\\
\email{g.zhou@curtin.edu.au}
}

\date{Received: date / Accepted: date}

\maketitle

\begin{abstract}
This paper addresses multilinear systems of equations which arise in various applications such as data mining and numerical partial differential equations. When the multilinear system under consideration involves a nonsingular $\mathcal{M}$-tensor and a nonnegative right-hand side vector, it may have multiple nonnegative solutions. In this paper, we propose an algorithm which can always preserve the nonnegativity of solutions. Theoretically, we show that the sequence generated by the proposed algorithm is a nonnegative decreasing sequence and converges to a nonnegative solution of the system. Numerical results further support the novelty of the proposed method. Particularly, when some elements of the right-hand side vector are zeros, the proposed algorithm works well while existing state-of-the-art solvers may not produce a nonnegative solution.




\end{abstract}

\section{Introduction}\label{Int}

Let $\mathbb{R}$ be the real field. A multidimensional array consisting of $n^m$ entries is called a real $m$-th order $n$-dimensional square tensor if we define it by
\begin{equation*}\label{defA}
{\cal A} = (a_{i_1 i_2 ... i_m}),\quad a_{i_1 i_2 \ldots i_m} \in \mathbb{R},
\quad 1\le i_1, i_2,\ldots, i_m \le n.
\end{equation*}
Throughout this paper, we suppose $m > 2$.
In what follows, we denote the set of all real tensors of order $m$ and dimension $n$ by
$\mathbb{T}_{m,n}$. Let $[n]:=\{1,2,\cdots,n\}$.
For a tensor ${\cal A}=(a_{i_1i_2\cdots i_m})\in\mathbb{T}_{m,n}$ and a vector $x=(x_1,x_2,\cdots,x_n)^\top \in\mathbb{R}^n$, we define ${\cal A}x^{m-1}\in \mathbb{R}^n$, whose $i$-th element is given by
\begin{equation}\label{tv1}
\left({\cal A}x^{m-1}\right)_i:=\sum^n_{i_2,\cdots,i_m=1}a_{ii_2\cdots i_m}x_{i_2}\cdots x_{i_m},~~\forall i\in [n],
\end{equation}
and ${\cal A}x^{m-2}\in \mathbb{R}^{n\times n}$, whose $(i,j)$-th element is given by
\begin{equation}\label{tv2}
\left({\cal A}x^{m-2}\right)_{ij}:=\sum^n_{i_3,\cdots,i_m=1}a_{iji_3\cdots i_m}x_{i_3}\cdots x_{i_m},~~\forall i,j\in [n].
\end{equation}

As defined independently in Qi \cite{Qi2005} and Lim \cite{Lim05}, we call $\lambda\in\mathbb{R}$ an eigenvalue and $x\in\mathbb{R}^n\setminus\{0\}$ the corresponding eigenvector of ${\cal A}$ if they satisfy the following equality:
\begin{equation*}\label{ep}
{\cal A}x^{m-1}=\lambda x^{[m-1]},
\end{equation*}
where $x^{[m-1]}\in\mathbb{R}^n$ is given by
$(x^{[m-1]})_i:=(x_i)^{m-1}, ~i=1,2,\cdots,n.$
The spectral radius $\rho({\cal A})$ of ${\cal A}$ is the maximum modulus of its eigenvalues, which is given by
\begin{equation*}\label{sr}
\rho({\cal A})=\max\{|\lambda|:\lambda\mbox{ is an eigenvalue of }{\cal A}\}.
\end{equation*}
Here, we refer the reader to the recent monograph \cite{QL2017} for more details of spectral theory on tensors.
Below, we recall the definition of ${\mathcal M}$-tensor.
\begin{definition}[\cite{DQW13,ZQZ14}]\label{ZMsMfun}
A tensor ${\cal A}\in\mathbb{T}_{m,n}$ is called
\begin{description}
\item{(i)}
a ${\mathcal Z}$-tensor if all of its off-diagonal entries are non-positive;
\item{(ii)}
an ${\mathcal M}$-tensor if it can be written as ${\cal A}=s{\cal I}-{\cal B}$, where ${\cal B}\geq0$ and $s\geq\rho({\cal B})$;
\item{(iii)}
a nonsingular ${\mathcal M}$-tensor if it is an ${\mathcal M}$-tensor with $s>\rho({\cal B})$.
\end{description}
\end{definition}

With the above preparations, the so-called multilinear system (a.k.a., tensor equations) refers to the task of finding a vector $x\in{\mathbb R}^n$ such that
\begin{equation} \label{mle}
{\cal A}x^{m-1}=b,
\end{equation}
where ${\cal A}\in\mathbb{T}_{m,n}$ and $b\in\mathbb{R}^n$. It has been verified that multilinear systems have many applications in data mining and numerical partial  differential equations, e.g.,
see \cite{BLNT13,DW16,LXX17,LN15,XJW18} for applications of this topic. Recently, many results in both theory and algorithm for \eqref{mle} have been developed in the literature, e.g., \cite{BaG13,BoVDDD18,DW16,Han17,HLQZ18,LXX17,LLV18b,LDG19,LZZ18,LLV18a,LM18,WCW19,XJW18}. In particular, some works are mainly contributed to a special case of \eqref{mle} where the coefficient tensor is an ${\mathcal M}$-tensor  \cite{ZQZ14} due to its widespread applications and promising properties (see \cite{DQW13,ZQZ14}). For example, Ding and Wei \cite{DW16} first showed that the multilinear system has a unique solution when the coefficient tensor ${\mathcal A}$ of \eqref{mle} is a nonsingular ${\mathcal M}$-tensor and $b$ is a positive vector. To find a solution to the underlying system \eqref{mle} with ${\mathcal M}$-tensors, some state-of-the-art algorithms, including the Jacobi, Gauss-Seidel, Newton methods \cite{DW16}, homotopy method (denoted by `HM') \cite{Han17}, Newton-Gauss-Seidel method \cite{LXX17} and tensor methods \cite{XJW18} for symmetric ${\mathcal M}$-tensors, tensor splitting methods \cite{LLV18b,LLV18a}, and the locally and quadratically convergent Newton-type algorithm (denoted by `QCA') for asymmetric tensors \cite{HLQZ18}, are proposed. Besides, there are some recent papers devoted to \eqref{mle} with other structured tensors, e.g., see \cite{LDG19,LZZ18,LM18}, and extended models of \eqref{mle}, e.g., \cite{BaG13,BoVDDD18,DZCQ18,LYHQ18,YLLH18}.

Indeed, most of the papers mentioned above paid attention to the case where \eqref{mle} has a positive vector $b$, and some algorithms are developed under the assumption that the coefficient tensor ${\mathcal A}$ is symmetric. However, some real-world problems may not possess such positivity on $b$ and symmetry property on ${\mathcal A}$, thereby possibly limiting the applicability of some algorithms. Naturally, the right-hand side vector $b$ being nonnegative (not necessarily positive) may be more general in some real-world problems or more useful for promoting a sparse solution (e.g., \cite{LN15,LQX15}) than the fully positive case. When $b$ is nonnegative, a good news from \cite{GLQX15} is that the system \eqref{mle} with a nonsingular ${\mathcal M}$-tensor has a nonnegative solution, but its solution may not be unique. Below, we take an example from \cite{GLQX15} to illustrate the nonuniqueness of solutions if $b$ is nonnegative but not positive.
\begin{example}\label{exam1}
Let ${\cal A}=(a_{i_1 i_2 i_3 i_4})\in \mathbb{T}_{4,2}$, where $a_{1111}=1$, $a_{2222}=1$, $a_{1112}=-2$ and all other $a_{i_{1}i_{2}i_{3}i_{4}}=0$. It has been proved in \cite{GLQX15} that ${\cal A}$ is a nonsingular $\mathcal{M}$-tensor. By the definition given in \eqref{tv1}, we immediately have
\begin{equation*}
{\cal A}x^3=\left(\begin{array}{c}
x_{1}^3-2x_1^2x_2 \\ x_{2}^3
\end{array}\right).
\end{equation*}
Let $b=(0,1)^\top \geq0$, then it is easy to see that both $x^*=(0,1)^\top $ and $x^\star=(2,1)^\top $ are  solutions of the system ${\cal A}x^{3}=b$.
\end{example}

Actually, we observe that a common feature of the numerical experiments presented in most of the existing tensor equation papers is that they only consider the case where $b$ is a fully positive vector. Therefore, a natural question is that do these algorithms still work for the case where $b$ is a nonnegative but not positive vector? If not, can we design some algorithms to handle such a case?

Taking the aforementioned questions, in this paper, we are interested in the multilinear system \eqref{mle} with a nonsingular (but not necessarily symmetric) ${\mathcal M}$-tensor ${\mathcal A}$ and a nonnegative (possibly with many zero components) vector $b$. Although it has been proved theoretically that such a system has one nonnegative solution (possibly not unique), there is an algorithmic gap. To our knowledge, it seems that no  algorithm is designed for the system \eqref{mle} with a nonnegative $b$. Therefore, we aim at introducing an efficient algorithm to solve \eqref{mle} with a nonsingular ${\mathcal M}$-tensor ${\mathcal A}$ and a nonnegative vector $b$, thereby filling the gap from algorithmic perspective. It is noteworthy that the proposed algorithm is well-defined in the sense that its iterative sequence is a nonnegative decreasing sequence and converges to a nonnegative solution of the system. Numerical experiments tell us that the proposed algorithm is efficient and can successfully find a nonnegative solution as long as the problem under consideration has a nonnegative solution and an appropriate starting point is taken. However, the state-of-the-art solvers, e.g., HM and QCA, tailored for \eqref{mle} with a positive vector $b$ may not produce a desired nonnegative solution in some situations. In addition, just after this article has been completed,
Li, Guan and Wang \cite{LGW19} proposed an algorithm for finding  a nonnegative solution of the multilinear system \eqref{mle}. It is shown that an increasing sequence is generated by the algorithm in  \cite{LGW19} and it converges to a nonnegative solution of the multilinear system \eqref{mle}, while our proposed algorithm produces a decreasing sequence.

The rest of this paper is organized as follows. In Section \ref{s2}, we briefly review some basic definitions and properties about structured tensors.  In Section \ref{s3}, we present a nonnegativity preserving algorithm for solving (\ref{mle}), and analyze the convergence of the proposed algorithm. In Section \ref{s4}, we report our numerical results to show the efficiency and novelty of the proposed algorithm.  Finally, we conclude the paper with some remarks in Section \ref{s5}.

We conclude this section with some notation and terminology. Throughout this paper, we use lowercases $x,y,z,\cdots$ for vectors, capital letters $A,B,C,\cdots$ for matrices and calligraphic letters ${\cal A}, {\cal B}, {\cal C},\cdots$ for tensors. We denote  $\mathbb{R}^n:=\{x=(x_1,x_2,\cdots,x_n)^\top :x_i\in \mathbb{R},\forall i \in [n]\}$,  $\mathbb{R}^n_+:=\{x\in\mathbb{R}^n:x\geq0\}$, where $x\geq0$ denotes $x_i\geq0$ for any $i\in[n]$, and $\mathbb{R}^n_{++}:=\{x\in\mathbb{R}^n:x>0\}$ where $x>0$ means $x_i>0$ for any $i\in[n]$. Suppose that $\theta$ is a subset of $[n]$, then $x_\theta\in\mathbb{R}^{|\theta|}$ represents the corresponding sub-vector of $x\in\mathbb{R}^n$, where $|\theta|$ denotes the cardinality of the set $\theta$, and $A_{\theta\theta}\in\mathbb{R}^{|\theta|\times |\theta|}$ represents the corresponding principal sub-matrix of $A\in\mathbb{R}^{n\times n}$. Besides, ${\cal I}=(\delta_{i_1\cdots i_m})\in {\mathbb T}_{m,n}$ denotes the identity tensor, where $\delta_{i_1\cdots i_m}$ is the Kronecker symbol
\begin{equation*}
\delta_{i_1\cdots i_m}=\left\{
\begin{array}{ll}
1,&\;\;{\rm if~}i_1=\cdots =i_m,\\
0,&\;\;{\rm otherwise},
\end{array}
\right.\quad 1\leq i_1,\cdots,i_m\leq n,\end{equation*}
and ${\cal A}\geq0$ denotes a nonnegative tensor, which means that all of its entries are nonnegative.
For a
continuously differentiable function $F:\mathbb{R}^n \to \mathbb{R}^n$, we denote
the Jacobian of $F$ at $x\in \mathbb{R}^n$ by $F^{\prime}(x)$.

\section{Preliminaries }\label{s2}
In this section, we briefly recall some definitions and properties about structured matrices and tensors, which will be used throughout this paper.

A tensor ${\cal A}\in\mathbb{T}_{m,n}$ is called a symmetric tensor if all its elements are invariant under arbitrary permutation of their indices, and it is called a semi-symmetric tensor with respect to the indices $\{i_2,\cdots,i_m\}$ if for any index $i\in[n]$, the $(m-1)$-order $n$-dimensional tensor ${\cal A}_i:=(a_{ii_2\cdots i_m})_{1\leq i_2,\cdots,i_m\leq n}$ is symmetric. Thus, from \cite{Han17} we know that, for any tensor ${\cal A}=(a_{i_1i_2\cdots i_m})$, there always exists a semi-symmetric tensor $\bar{{\cal A}}=(\bar{a}_{i_1i_2\cdots i_m})$, denoted by
\begin{equation}\label{ss}
\bar{a}_{i_1i_2\cdots i_m} = \frac{1}{(m-1)!}\sum_{\pi}a_{i_1\pi(i_2\cdots i_m)},
\end{equation}
such that ${\cal A}x^{m-1}=\bar{{\cal A}}x^{m-1}$  and $({\cal A}x^{m-1})'=(m-1)\bar{{\cal A}}x^{m-2}$ for any $x\in\mathbb{R}^n$, where the sum is over all the permutations $\pi(i_2\cdots i_m)$.

From Definition \ref{ZMsMfun}, we know that (nonsingular) ${\mathcal M}$-tensor is a generalization of (nonsingular) $M$-matrix. Below, we recall some properties of nonsingular $M$-matrices, which will be used in the later analysis.

\begin{theorem}[\cite{BP94}]\label{Mm1}
Let $A\in\mathbb{R}^{n\times n}$ be a $Z$-matrix, then the following statements are equivalent:
\begin{description}
\item{\rm(i)} $A$ is a nonsingular $M$-matrix;
\item{\rm(ii)} There exists an $x\in\mathbb{R}^n_{++}$ satisfying $Ax\in\mathbb{R}^n_{++}$;
\item{\rm(iii)} $A^{-1}$ exists and $A^{-1}$ is a nonnegative matrix.
\end{description}
\end{theorem}
Similarly, some properties of nonsingular ${\mathcal M}$-tensors are shown below.
\begin{theorem}[\cite{DQW13}]\label{nmt}
Let ${\cal A}\in\mathbb{T}_{m,n}$ be a ${\cal Z}$-tensor, then the following statements are equivalent:
\begin{description}
\item{\rm(i)} ${\cal A}$ is a nonsingular ${\mathcal M}$-tensor;
\item{\rm(ii)} There exists an $x\in\mathbb{R}^n_{++}$ satisfying ${\cal A}x^{m-1}\in\mathbb{R}^n_{++}$;
\item{\rm(iii)} All diagonal entries of ${\cal A}$ are positive and there exists a positive diagonal matrix $D\in\mathbb{R}^{n\times n}$ such that ${\cal A}D^{m-1}$ is strictly diagonally dominated, where
    \begin{equation*}\label{d3}
    {\cal A}D^{m-1}={\cal A}\times_2D\times_3\cdots\times_mD,
    \end{equation*}
    and
    \begin{equation*}\label{d4}
    \left({\cal A}\times_k D\right)_{i_1\cdots i_{k-1}j_ki_{k+1}\cdots i_m}=\sum^n_{i_k=1}a_{i_1\cdots i_m}x_{i_kj_k},~~\forall k\in[m]
    \end{equation*}
    with $x_{i_kj_k}=D_{i_kj_k}$.
\end{description}
\end{theorem}

Let ${\cal A}\in\mathbb{T}_{m,n}$ and $\bar{{\cal A}}\in\mathbb{T}_{m,n}$ be the corresponding semi-symmetric tensor of ${\cal A}$ which satisfies (\ref{ss}). Then, we can obtain the following relationship between these two tensors.
\begin{lemma}[\cite{Han17}]\label{ssn}
Let ${\cal A}\in\mathbb{T}_{m,n}$ be a nonsingular ${\mathcal M}$-tensor. Then, $\bar{{\cal A}}$ is also a nonsingular ${\mathcal M}$-tensor.
\end{lemma}

Based on the above lemma, we can further conclude the following result which is vital for the convergence analysis of the proposed algorithm in this paper.
\begin{lemma}\label{ssjn}
Suppose that ${{\cal A}}\in\mathbb{T}_{m,n}$ is a nonsingular ${\mathcal M}$-tensor and there exists an $\hat{x}\in\mathbb{R}^n_{+}$ satisfying ${{\cal A}}\hat{x}^{m-1}\in\mathbb{R}^n_{+}$.
Let $\bar{I} = \{i: \left({\cal A}(\hat{x})^{m-1}\right)_i > 0\}$ and suppose $\bar{I} \ne \emptyset$.
Then, for any nonempty subset $I$ of $\bar{I}$,  $({{\cal A}}\hat{x}^{m-1} )^{\prime}_{II}$ is a nonsingular $M$-matrix.
Here,  $({{\cal A}}\hat{x}^{m-1} )^{\prime}_{II}$ is a principal sub-matrix of $({{\cal A}}\hat{x}^{m-1} )^{\prime}$.
\end{lemma}
\begin{proof}
Let $\bar{{\cal A}}\in\mathbb{T}_{m,n}$ be the corresponding semi-symmetric tensor of ${\cal A}$ which satisfies (\ref{ss}). Then, ${{\cal A}}\hat{x}^{m-1} = {\bar{\cal A}}\hat{x}^{m-1}$,  $({{\cal A}}\hat{x}^{m-1} )^{\prime} = (m-1)\bar{{\cal A}}\hat{x}^{m-2}$ and
the $(i,j)$-th entry of the matrix $(m-1)\bar{{\cal A}}\hat{x}^{m-2}$ is given by
\begin{equation*}\label{jz}
\left((m-1)\bar{{\cal A}}\hat{x}^{m-2}\right)_{ij}=(m-1)\sum^n_{i_3,\cdots,i_m=1}\bar{a}_{iji_3\cdots i_m}\hat{x}_{i_3}\cdots \hat{x}_{i_m},~~\forall i,j\in [n].
\end{equation*}
By Lemma \ref{ssn},  $\bar{{\cal A}}$ is a nonsingular ${\mathcal M}$-tensor. Thus,  when $i\neq j$, $\bar{a}_{iji_3\cdots i_m}\leq0$. Hence, $(m-1)\bar{{\cal A}}\hat{x}^{m-2}$ is a $Z$-matrix.  So,
$({{\cal A}}\hat{x}^{m-1} )^{\prime}_{II} =\left [ (m-1)\bar{{\cal A}}\hat{x}^{m-2}\right ]_{II}$ is a $Z$-matrix.
Since $\left({\cal A}(\hat{x})^{m-1}\right)_i > 0$ for any $i\in I$, we have $\hat{x}_I > 0$.
Thus, we obtain
\begin{align}
\left({{\cal A}}\hat{x}^{m-1} \right)^{\prime}_{II} \cdot \hat{x}_{I}
& = (m-1) \left [ {\bar{\cal A}}\hat{x}^{m-2}\right ]_{II} \cdot \hat{x}_{I}   \nonumber \\
&\geq(m-1)\left({\bar {\cal A}}(\hat{x})^{m-1}\right)_{I} \nonumber \\
& = (m-1)\left({\cal A}(\hat{x})^{m-1}\right)_{I} \nonumber \\
&>0, \nonumber
\end{align}
where we use `$\cdot$' to represent the matrix-vector product.
Thus, it follows from Theorem \ref{Mm1}  that $({{\cal A}}\hat{x}^{m-1} )^{\prime}_{II}$ is a nonsingular $M$-matrix.
\qed\end{proof}

At the end of this section, we recall two important results on \eqref{mle}, which guarantee that the solution set of the problem under consideration is nonempty.
\begin{theorem}[\cite{DW16}]\label{nuss}
Let ${\cal A}\in\mathbb{T}_{m,n}$ be a nonsingular ${\mathcal M}$-tensor and $b\in\mathbb{R}^n_{++}$. Then, the system ${\cal A}x^{m-1}=b$ has a unique positive solution.
\end{theorem}
\begin{theorem}[\cite{GLQX15}]\label{nus}
Let ${\cal A}\in\mathbb{T}_{m,n}$ be a nonsingular ${\mathcal M}$-tensor and $b\in\mathbb{R}^n_+$. Then, the system ${\cal A}x^{m-1}=b$ has a nonnegative solution.
\end{theorem}

\section{Algorithm and Convergence Analysis}\label{s3}
In this section, we are going to present a Newton-type method which can always preserve the nonnegativity of the iterative sequence for the system \eqref{mle} with a nonsingular ${\mathcal M}$-tensor ${\cal A}$ and a nonnegative right-hand side vector $b$. We will also state that this algorithm is well-defined and converges to a solution of the system.

For notational convenience, we first define $F:\mathbb{R}^n\rightarrow\mathbb{R}^n$ by
\begin{equation}\label{f}
F(x)={\cal A}x^{m-1}-b.
\end{equation}
Then, the system (\ref{mle}) is equivalent to $F(x)=0$. Furthermore, we have
\begin{equation*}\label{fj}
F'(x)=(m-1)\bar{{\cal A}}x^{m-2},
\end{equation*}
where $\bar{{\cal A}}$ is the corresponding semi-symmetric tensor, which satisfies (\ref{ss}), of the tensor ${\cal A}$.

Hereafter, we describe details of the new algorithm for solving the system (\ref{mle}) in Algorithm \ref{alg31}.

\begin{algorithm}[!htbp]
\caption{(A Nonnegativity Preserving Algorithm for \eqref{mle}).}\label{alg31}
\begin{algorithmic}[1]
\STATE Let $x^0\in \mathbb{R}^n_{+}$ satisfying $F(x^0)\geq0$ be a starting point. Choose $\delta_1,  \delta_2 \in (0, 1)$.
\WHILE{$\| F(x^k)\|\neq 0$}
\STATE Let $j\in [n]$ satisfying $F_j (x^k) = \max _{i\in [n]} \{ F_i (x^k)\}$,  $I_k=\{i: F_i (x^k) = \left({\cal A}(x^k)^{m-1}\right)_i - b_i > 0\} \setminus \{j\}$  and $\bar{I}_k =\{\ell: F_\ell(x^k) = \left({\cal A}(x^k)^{m-1}\right)_\ell- b_\ell= 0\}$.
\STATE Let $x^{k+1}_{\bar I^k} := x^k_{\bar I^k}$
\STATE Let $\bar{d}^k_{j}= - x^k_j$.
\FOR{$p=0,1,\cdots$}
\STATE Compute $g(p) = F_j(x_1^k, \cdots , x_{j-1}^k, x^k_j + \delta_1^p \bar{d}^k_{j}, x_{j+1}^k, \cdots, x_{n}^k).$
\STATE If $g(p) \ge 0$, then let $x^{k+1}_j := x^k_j + \delta_1^p \bar{d}^k_{j}$ and stop.
\ENDFOR
\STATE Let $d^k\in\mathbb{R}^n$, where $d^k_{\bar{I}_k}=0$, $d^k_{j}=0$ and
\begin{equation}\label{a1}
d^k_{I_k}=-\left[F'(x^k)\right]^{-1}_{I_kI_k}\left[F(x^k)\right]_{I_k}.
\end{equation}
\FOR{$q=0,1,\cdots$}
\STATE Comput $G(q) = F_{I_k}( x^k+ \delta_2^q d^k).$
\STATE If $G(q) \ge 0$, then let $x^{k+1}_{I_k} := x^k_{I_k} + \delta_2^q d^k_{I_k}$ and stop.
\ENDFOR
\ENDWHILE
\end{algorithmic}
\end{algorithm}

\begin{remark}\label{rem1}
 Notice that Algorithm \ref{alg31} needs an initial point $x^0$ satisfying $x^0\in {\mathbb R}^n_{+}$ and ${\cal A}(x^0)^{m-1}\geq b$, which can not be guaranteed for an arbitrary $x^0\in{\mathbb R}^n_+$. Fortunately, we can obtain
 such a starting point $x_0$ for Algorithm \ref{alg31} as follows.
 Let $\Gamma$ be an index set with respect to zero components of $b$, i.e., ${\Gamma}=\{i: b_i=0, i=1,2,\cdots,n\}$,  and $\epsilon\in{\mathbb R}^{n}$ with  $\epsilon_{i}=10^{-3}, i\in \Gamma$ and $\epsilon_{j}=0, j \notin \Gamma$.
By the nonsingularity of the ${\mathcal M}$-tensor ${\mathcal A}$ in \eqref{mle}, we can easily obtain  a unique positive solution (e.g., see \cite{DW16}) of the following perturbed system:
 \begin{equation}\label{pmle}
 {\mathcal A}x^{m-1}=  b + \epsilon.
 \end{equation}
 Clearly, the unique positive solution $\tilde x$ of \eqref{pmle} implies that $ {\mathcal A}{\tilde x}^{m-1}-b = \epsilon\geq 0$.
Therefore, we can employ the state-of-art solvers (e.g., \cite{DW16,Han17,HLQZ18}) to find the unique positive solution $\tilde x$ of \eqref{pmle} satisfying $F(\tilde x)=\epsilon\geq 0$ and let $\tilde x$ be a starting point $x^0$ of Algorithm \ref{alg31}.
\end{remark}

\begin{remark}\label{rem2}
Notice that Step 10 in Algorithm \ref{alg31} is indeed a Newton step. It can be easily seen from the definition of the index set $I_k$ that the subproblem \eqref{a1} is well defined in the sense that $\left[F'(x^k)\right]_{I_kI_k}$ is always a nonsingular matrix (see Lemma \ref{0}). If the scale of $I_k$ is relatively small, we can directly compute the inverse of $\left[F'(x^k)\right]_{I_kI_k}$ and gainfully obtain the accurate solution of \eqref{a1}.
\end{remark}

Next, we will present a convergence analysis of Algorithm \ref{alg31}, in addition to showing that the proposed algorithm has some promising theoretical properties.
In what follows, we always suppose that ${{\cal A}}\in\mathbb{T}_{m,n}$ is a nonsingular ${\mathcal M}$-tensor.

\begin{lemma}\label{0}
In Algorithm \ref{alg31}, for any $k$, $\left[F'(x^k)\right]_{I_kI_k}$ is a nonsingular $M$-matrix.
\end{lemma}
\begin{proof}
Let $\bar{{\cal A}}\in\mathbb{T}_{m,n}$ be the corresponding semi-symmetric tensor of ${\cal A}$. Then, ${{\cal A}}{(x^k)}^{m-1} = {\bar{\cal A}}{(x^k)}^{m-1}$ and $({{\cal A}}{(x^k)}^{m-1} )^{\prime} = (m-1)\bar{{\cal A}}{(x^k)}^{m-2}$.
Thus,
\begin{equation*}\label{l02}
\left[F'(x^k)\right]_{I_kI_k}  = (m-1)\left [ \bar{{\cal A}}{(x^k)}^{m-2}\right ]_{I_kI_k}.
\end{equation*}
It follows from Lemma \ref{ssjn} that $\left[F'(x^k)\right]_{I_kI_k}$ is a nonsingular $M$-matrix.
\qed\end{proof}

\begin{lemma}\label{00}
At the $k$-th iteration of Algorithm \ref{alg31},
let $x^k_j (t) = x^k_j + t \bar{d}^k_{j}$ and $g(t) = F_j(x_1^k, \cdots , x_{j-1}^k, x^k_j + t \bar{d}^k_{j}, x_{j+1}^k, \cdots, x_{n}^k).$ Then, we have the following results:
\begin{description}
\item[\rm (i).]
$x^k_j (t) \ge 0$ for any $t\in (0, 1]$.
\item[\rm (ii).]
There exists a positive number $\bar t \in(0,1]$ such that $g(t) \geq0$ for any $t\in (0, \bar t \ ]$.
\end{description}
\end{lemma}
\begin{proof} (i). Since $F_{j} (x^k)=({\cal A}(x^k)^{m-1} - b)_j >0$, clearly, we have $x^k_j > 0$. Then,
$x^k_j (t) =  x^k_j + t \bar{d}^k_{j} = (1-t) x^k_j \ge 0$ for any $t\in (0, 1]$.

(ii). First, we define
\begin{equation}\label{l321}
F_{jj} (x^k) :=  \partial F_j(x^k) / \partial x_j,
\end{equation}
where $j$ is given in Step 3 of Algorithm \ref{alg31}. Then, from Lemma \ref{ssjn}, we obtain that
\begin{equation}\label{l322}
F_{jj} (x^k)= (m-1)\left [ \bar{{\cal A}}{(x^k)}^{m-2}\right ]_{jj}>0.
\end{equation}

Since $\bar{d}^k_{j} = -x^k_j <0$, we have $x^k_j + t \bar{d}^k_{j}<x^k_j$ for any $t\in (0, 1]$. Let $A:={\bar{\cal A}}(x^k)^{m-2}$. Then, by a simple computation, we have $g(0) = F_j (x^k) > 0$ and
    \begin{align}
      g^{\prime}(t)
        &= F_{jj}(x_1^k, \cdots , x_{j-1}^k, x^k_j + t \bar{d}^k_{j}, x_{j+1}^k, \cdots, x_{n}^k) \bar{d}^k_{j}.  \nonumber
    \end{align}
Hence,
\begin{equation*}\label{l323}
g^{\prime}(0)  =  F_{jj}(x^k)  \bar{d}^k_{j} = (m-1)A_{jj}  \bar{d}^k_{j} < 0,
\end{equation*}
which implies that there exists a positive number $\bar t \in(0,1]$ such that $g(t) \geq0$ for any $t\in (0, \bar t \ ]$.
\qed\end{proof}

Similar to Lemma \ref{00}, we have the following lemma.

\begin{lemma}\label{1}
At the $k$-th iteration of Algorithm \ref{alg31},
let $x^k(\lambda) = x^k + \lambda d^k$ and $G(\lambda) = F_{I_k}(x^k+\lambda d^k)$. Then, we have the following results:
\begin{description}
\item[\rm (i).]
$x^k(\lambda)\geq0$ for any $\lambda\in (0, 1]$.
\item[\rm (ii).]
There exists a positive number $\bar \lambda \in(0,1]$ such that $G(\lambda) \geq0$ for any $\lambda\in (0, \bar \lambda \ ]$.
\end{description}
\end{lemma}
\begin{proof} (i). First, Lemma \ref{0} tells us that $\left[F'(x^k)\right]_{I_kI_k}$ is a nonsingular $M$-matrix, which together with Theorem \ref{Mm1} implies that $\left[F'(x^k)\right]^{-1}_{I_kI_k}$ is a nonnegative matrix. Thus, it immediately follows from \eqref{a1} that $d^k_{I_k}<0$.

Let $A:={\bar{\cal A}}(x^k)^{m-2}$ and $y:={\cal A}(x^k)^{m-1}$ for notational simplicity. By Definition \ref{ZMsMfun} and Lemma \ref{ssjn}, since $A_{I_kI_k}=\left[F'(x^k)\right]_{I_kI_k}$ is a nonsingular $M$-matrix, its diagonal entries are positive and off-diagonal entries are non-positive. Thus, it follows from \eqref{a1} and Step 13 of Algorithm \ref{alg31} that for any $\lambda\in (0,1]$,
\begin{align}
  x^k_{I_k}(\lambda)
  &=x^k_{I_k}+\lambda d^k_{I_k}=x^k_{I_k}-\frac{\lambda }{m-1}\left(A_{I_kI_k}\right)^{-1}\cdot(y-b)_{I_k} \nonumber\\
  &=\left(A_{I_kI_k}\right)^{-1}\cdot\left(A_{I_kI_k}\cdot x^k_{I_k}-\frac{\lambda}{m-1}y_{I_k}+\frac{\lambda}{m-1}b_{I_k}\right) \nonumber\\
  &\geq \left(A_{I_kI_k}\right)^{-1}\cdot\left(y_{I_k}-\frac{\lambda}{m-1}y_{I_k}+\frac{\lambda}{m-1}b_{I_k}\right) \nonumber\\
  &>0. \nonumber
\end{align}
Besides, both $d^k_{\bar{I}_k}=0$ and $d^k_j=0$ together with Lemma \ref{ssjn} lead to the truth that
\begin{equation*}\label{l112}
  x^{k}_{\bar{I}_k\cup\{j\}}(\lambda)=x^{k}_{\bar{I}_k\cup\{j\}}\geq0.
\end{equation*}
Then, for any number $\lambda\in(0,1]$, we have $x^{k}(\lambda)\geq0$.

(ii). First, for the convenience of notation, without loss of generality, we assume that $I_n$ consists of the first $|I_n|$ elements in $[n]$. Then, for any $i\in I_k$ we have
\begin{align}\label{l11}
  G_i(\lambda)
  &=F_i(x^k(\lambda))\nonumber\\
  &=a_{i\cdots i}\left(x^k_i(\lambda)\right)^{m-1}+\sum_{(i_2,\cdots,i_m)\neq(i,\cdots,i)}a_{ii_2\cdots i_m}x^k_{i_2}(\lambda)\cdots x^k_{i_m}(\lambda)-b_i, \nonumber
\end{align}
and
\begin{equation*}\label{l11}
F_i(x^k)=a_{i\cdots i}\left(x^k_i\right)^{m-1}+\sum_{(i_2,\cdots,i_m)\neq(i,\cdots,i)}a_{ii_2\cdots i_m}x^k_{i_2}\cdots x^k_{i_m}-b_i>0.
\end{equation*}
Since ${\cal A}$ is a nonsingular $\mathcal{M}$-tensor, all of its diagonal entries are positive and off-diagonal entries are non-positive. Hence, we know that
$$a_{i\cdots i}\left(x^k_i(\lambda)\right)^{m-1}>0\quad,\quad\sum_{(i_2,\cdots,i_m)\neq(i,\cdots,i)}a_{ii_2\cdots i_m}x^k_{i_2}(\lambda)\cdots x^k_{i_m}(\lambda)\leq0.$$
and
$$a_{i\cdots i}\left(x^k_i\right)^{m-1} > 0\quad,\quad\sum_{(i_2,\cdots,i_m)\neq(i,\cdots,i)}a_{ii_2\cdots i_m}x^k_{i_2}\cdots x^k_{i_m}\leq0,$$
What is more, we have $x^k_i(\lambda)<x^k_i$. Then,
\begin{equation}\label{l32}
  a_{i\cdots i}\left(x^{k}_i(\lambda)\right)^{m-1}<a_{i\cdots i}\left(x^k_i\right)^{m-1},
\end{equation}
and at the same time,
\begin{equation}\label{l12}
  \sum_{(i_2,\cdots,i_m)\neq(i,\cdots,i)}a_{ii_2\cdots i_m}x^{k}_{i_2}(\lambda)\cdots x^{k}_{i_m}(\lambda)\geq\sum_{(i_2,\cdots,i_m)\neq(i,\cdots,i)}a_{ii_2\cdots i_m}x^k_{i_2}\cdots x^k_{i_m}.
\end{equation}
Hence, by adjusting the value of $\lambda\in(0,1]$, we can guarantee that the sum of the left-hand sides of \eqref{l32} and \eqref{l12} is nonnegative, i.e., we can find a suitable number $\bar{\lambda}\in(0,1]$ such that $G(\lambda)\geq0$ for any $\lambda\in(0,{\bar \lambda}]$.
\qed\end{proof}

\begin{lemma}\label{11}
Algorithm \ref{alg31} is well-defined.
At the $k$-th iteration of Algorithm \ref{alg31}, we obtain two nonnegative integers $p^k$ and $q^k$ such that
\[
x^{k+1}_j := x^k_j + \delta_1^{p^k} \bar{d}^k_{j} \ge 0, \  x^{k+1}_{I_k} := x^k_{I_k} + \delta_2^{q^k} d^k_{I_k} > 0, \;\;
{\rm and}\;\;x^{k+1}_{\bar{I}_k} := x^k_{\bar{I}_k}.
\]
Additionally, we have $F(x^{k+1}) \ge 0$.
\end{lemma}
\begin{proof}
From Lemmas \ref{00} and \ref{1}, there exist two nonnegative integers $p^k$ and $q^k$ such that
\begin{equation*}\label{l111}
x^{k+1}_j := x^k_j + \delta_1^{p^k} \bar{d}^k_{j} \ge 0\quad{\rm and}\quad x^{k+1}_{I_k} := x^k_{I_k} + \delta_2^{q^k} d^k_{I_k} > 0.
\end{equation*}
Thus, Algorithm \ref{alg31} is well-defined. Since $\bar{d}^k_{j} < 0$ and $d^k_{I_k} < 0$, we have
\begin{equation*}\label{l112}
0\leq x^{k+1}\leq x^k.
\end{equation*}
${\cal A}$ is a nonsingular $\mathcal{M}$-tensor, so all of its diagonal entries are positive and off-diagonal entries are non-positive. We obtain the following:
\begin{numcases}{}\label{l113}
F_{\bar{I}_k}(x^{k+1}) \ge F_{\bar{I}_k}(x^{k}) = 0, \nonumber \\
F_{j}(x^{k+1}) \ge F_j(x_1^k, \cdots , x_{j-1}^k, x^{k+1}_j, x_{j+1}^k, \cdots, x_{n}^k) \ge 0, \nonumber\\
F_{{I}_k}(x^{k+1}) \ge F_{{I}_k}(x^k_{\bar{I}_k}, x^k_j,  x^{k+1}_{I_k}) \ge 0. \nonumber
\end{numcases}
Hence, we conclude $F(x^{k+1}) \ge 0$.
\qed\end{proof}

\begin{lemma}\label{33}
In Algorithm \ref{alg31}, for any $k$, we always have $J^k\subseteq J^{k+1}$, where $J^k:=\{j:F_j(x^k)=0,x^k_j=0\}$.
\end{lemma}
\begin{proof}
For any $i\in J^k$, we have $x^k_i=0$ and $F_i(x^k)=0$, that is,
\begin{equation*}\label{l331}
a_{i\cdots i}(x^k_i)^{m-1}+\sum_{(i_2,\cdots,i_m)\neq(i,\cdots,i)}a_{ii_2\cdots i_m}x^k_{i_2}\cdots x^k_{i_m}-b_i=0.
\end{equation*}
Since ${\cal A}$ is a nonsingular ${\cal M}$-tensor, its diagonal entries and off-diagonal entries are positive and non-positive, respectively. Combining this with the fact that $x^k\geq0$ and $x^k_i=0$, we can obtain that $b_i=0$. From Lemma \ref{11} we know that $0\leq x^{k+1}\leq x^k$ and $F(x^{k+1})\geq0$. Then, $x^{k+1}_i=0$, and
\begin{align*}\label{l332}
F_i(x^{k+1})
&=a_{i\cdots i}(x^{k+1}_i)^{m-1}+\sum_{(i_2,\cdots,i_m)\neq(i,\cdots,i)}a_{ii_2\cdots i_m}x^{k+1}_{i_2}\cdots x^{k+1}_{i_m}-b_i\\
&=\sum_{(i_2,\cdots,i_m)\neq(i,\cdots,i)}a_{ii_2\cdots i_m}x^{k+1}_{i_2}\cdots x^{k+1}_{i_m}\geq0.
\end{align*}
Thus we have $\sum_{(i_2,\cdots,i_m)\neq(i,\cdots,i)}a_{ii_2\cdots i_m}x^{k+1}_{i_2}\cdots x^{k+1}_{i_m}=0$, that is, $F_i(x^{k+1})=0$, which, together with $x^{k+1}_i=0$, implies that $i\in J^{k+1}$. Then, $J^k\subseteq J^{k+1}$.
\qed\end{proof}

\begin{theorem}\label{2}
The sequence $\{x^k\}$ generated by Algorithm \ref{alg31} is a decreasing sequence: $0\leq x^{k+1}\leq x^k$ for all $k$, and $x^k\rightarrow x^*$ as $k\rightarrow\infty$, where $x^*\in\mathbb{R}^n_{+}$.
In particular, if $b_i$, $i\in [n]$, is positive, then $x_i^*$ is also positive.
\end{theorem}
\begin{proof}
From Lemma \ref{11},  we know that  $0\leq x^{k+1}\leq x^k$ and $F(x^{k}) \ge 0$ for each $k$. Hence, the sequence $\{x^k\}$ is a decreasing sequence. Moreover, because $x^k\geq0$ for all $k$, the sequence $\{x^k\}$ is lower bounded, which implies that there exists a vector $x^*\in\mathbb{R}^n_+$ such that, as $k\rightarrow\infty$, $x^k\rightarrow x^*$ and ${\cal A}(x^*)^{m-1}\geq b$.

If $b_i$, $i\in [n]$, is positive, now we will show that $x_i^*$ is also positive.
Assume on the contrary that $x_i^*$ is zero.
Since ${\cal A}$ is a nonsingular ${\mathcal M}$-tensor, all of its off-diagonal entries are non-positive. Thus,
\begin{equation*}\label{2e}
\left({\cal A}(x^*)^{m-1}-b\right)_i=\sum^n_{i_2,\cdots,i_m=1}a_{ii_2\cdots i_m}x^*_{i_2}\cdots x^*_{i_m}-b_i<0,
\end{equation*}
which is a contradiction. Hence, $x_i^*$ is positive.
\qed\end{proof}

\begin{theorem}\label{3}
Let $x^*$ be the limit point of $\{x^k\}$ generated by Algorithm \ref{alg31}. Then, $x^*$ is a solution of $F(x)=0$.
\end{theorem}
\begin{proof}
Suppose on the contrary, the limit point $x^*$ is not a solution of the equation $F(x)=0$.
Let
$$I^*=\{i: \left({\cal A}(x^*)^{m-1}\right)_i - b_i > 0\}\quad{\rm and}\quad \bar{I}^* =\{j: \left({\cal A}(x^*)^{m-1}\right)_j - b_j = 0\}.$$
Suppose $F_i(x^*) = \max_{j\in I^*} \{\left({\cal A}(x^*)^{m-1}\right)_j - b_j\}$.
Let $\bar{d}^*_i = -x^*_i$ and
\begin{equation*}\label{t31}
g_i(t) = F_i(x_1^*, \cdots , x_{i-1}^*, x^*_i + t \bar{d}^*_{i}, x_{i+1}^*, \cdots, x_{n}^*).
\end{equation*}
Let $F_{ii}(x^*)$ be defined by \eqref{l321}. From \eqref{l322}, we have $F_{ii} (x^*) > 0$. Hence,
$$g^{\prime}_i(0) = F_{ii}(x^*) \bar{d}^*_{i} < 0\quad{\rm and}\quad g_i(0) > 0.$$
Therefore, there exists a $\bar{t} \in (0, 1)$ such that $g_i(t) > 0$ for all $t \in (0, \bar{t} \ ]$.
This implies that there exists a nonnegative integer $p^*$ such that
\begin{equation*}\label{t32}
F_i(x_1^*, \cdots , x_{i-1}^*, (1- \delta_1^{p^*})x^*_i, x_{i+1}^*, \cdots, x_{n}^*) > 0.
\end{equation*}
Hence, there exists a small neighbourhood $N(x^*, \epsilon), \epsilon > 0$, of $x^*$ such that for any $x \in N(x^*, \epsilon)$, we have $F_i(x_1, \cdots , x_{i-1}, (1- \delta_1^{p^*})x_i, x_{i+1}, \cdots, x_{n}) > 0$. Since $x^*$ is the limit point of the sequence $\{x^k\}$ generated by Algorithm \ref{alg31},
for all sufficiently large $k$,
\begin{equation*}\label{t34}
F_i(x_1^k, \cdots , x_{i-1}^k, (1- \delta_1^{p^*})x_i^k, x_{i+1}^k, \cdots, x_{n}^k) > 0.
\end{equation*}
Thus, we obtain that
\begin{equation*}\label{t35}
x_i^{k+1} = (1- \delta_1^{p^k})x_i^k \le (1- \delta_1^{p^*})x_i^k
\end{equation*}
holds for all sufficiently large $k$. Therefore, $x^*_i \le (1- \delta_1^{p^*})x_i^*$. Since $x_i^* > 0$, this contradiction implies that the limit point $x^*$ is a solution of $F(x)=0$.
\qed\end{proof}

For the system \eqref{mle} equipped with a nonnegative but not positive vector $b$, Theorem \ref{3} shows our proposed algorithm is always globally convergent, however it is unclear, to our knowledge, whether or not the existing algorithms such as the ones in \cite{Han17,HLQZ18}   have this convergence property. In the next section, we will
report our numerical results to show our proposed algorithm is efficient  when the right-hand side is nonnegative but not positive.

\section{Numerical Experiments}\label{s4}

In this section, we will show the numerical performance of Algorithm \ref{alg31} (denoted by `NPA') on the multilinear system \eqref{mle} by implementing it in {\sc Matlab}. Apart from this, we will also compare our proposed algorithm with the {\it homotopy method} (denoted by `HM') proposed by Han in \cite{Han17}, whose code can be downloaded from Han's homepage\footnote{http://homepages.umflint.edu/$\sim$lxhan/software.html}, the {\it globally and quadratically convergent algorithm} (denoted by `QCA') proposed by He {\it et al.} in \cite{HLQZ18}, and the {\sc Matlab} script in the optimization toolbox `\verb"lsqnonlin"' (denoted by `NLSQ'), which is devoted to finding solutions of nonlinear least square problems. An updated version of NLSQ has been used in \cite{BoVDDD18} to solve general multilinear systems. All numerical experiments are done in {\sc Matlab} R2014a on a workstation computer with Intel(R) Xeon(R) CPU E5-2630 $@$2.20GHz and 128 GB memory running Microsoft Windows 10. Throughout, we employ the tensor toolbox \cite{TensorT} to compute tensor-vector products as well as semi-symmetrization of tensors.

The main contribution of the paper is a nonnegativity preserving algorithm customized for the multilinear system with a nonnegative but not positive right-hand side $b$. Correspondingly, we will show that Algorithm \ref{alg31} is reliable for the case with a nonnegative vector $b$, while HM \cite{Han17}, QCA \cite{HLQZ18} and NLSQ may not produce a nonnegative (especially sparse) solution in some scenarios.

Algorithm \ref{alg31} was implemented as follows. Given a starting point  $x^0$, if $F(x^0)\geq0$, then generate a decreasing sequence $\{x^k\}$ by Algorithm \ref{alg31} directly.  If the condition $F(x^0)\geq0$  is not satisfied, we first obtain a point
$\tilde x$ by QCA in \cite{HLQZ18} such that $F(\tilde x)\geq0$ and then generate a decreasing sequence $\{x^k\}$ by Algorithm \ref{alg31} starting from $\tilde x$ (see details in Remark \ref{rem1}).
We notice that Steps 6-9 and 11-14 of  Algorithm \ref{alg31} are simplest Armijo line search procedures. In our experiments, we set $\delta_1=0.2$ and $\delta_2=0.5$ to update the next iterate $x^{k+1}$.
Additionally, the subproblem \eqref{a1} is a dimensionality reduced linear system, where the coefficient matrix $[F'(x^{k})]_{I_kI_k}$ is always a nonsingular M-matrix. So, we will solve such a subproblem directly by the `\verb"left matrix divide: \"' (the multiplication of the inverse of a matrix and a vector) if the dimension $n$ is strictly less than $20$, i.e., $n<20$; otherwise, we solve the well-defined subproblem \eqref{a1} by the {\sc Matlab} script `\verb"pcg"' (preconditioned conjugate gradient method) as used in \cite{HLQZ18}.

The same as the algorithms proposed in \cite{Han17} and \cite{HLQZ18}, we solve the scaled system
\begin{equation*}\label{ne4}
\tilde{{\cal A}}x^{m-1}=\tilde{b}
\end{equation*}
of the original one when comparing the three algorithms, where $\tilde{{\cal A}}=\frac{1}{\kappa}{\cal A}$, $\tilde{b}=\frac{1}{\kappa}b$ and $\kappa$ is the largest absolute value among the values of the entries of ${\cal A}$ and $b$. Besides, as used (suggested) in \cite{Han17,HLQZ18}, the stopping criterion for HM, QCA, and NPA is defined by
\begin{equation}\label{ne5}
{\textsf{ReErr}}:=\|\tilde{{\cal A}}(x^k)^{m-1}-\tilde{b}\|_2\leq {\rm Tol},
\end{equation}
here ${\rm Tol}>0$ is a preset tolerance. For the parameters of QCA, we follow the settings as used in \cite{HLQZ18}, i.e., $\delta=0.5,\gamma=0.8, \sigma=0.2$, and $\bar t = 2/(5\gamma)$. In the following, we will show the efficiency of NPA  for finding a nonnegative solution to \eqref{mle} through experiments with synthetic data.

\begin{example}\label{exam2}
This example is a modified version of Example \ref{exam1} in Introduction. Here, we use the same tensor ${\mathcal A}$  described in Example \ref{exam1}. Clearly, (i). when we take the right-hand side vector $b$ as $b=(0,\zeta^3)^\top$, where $\zeta$ is a nonnegative number, the resulting multilinear system has two solutions $(0,\zeta)^\top$ and $(2\zeta,\zeta)^\top$; (ii). when we set the right-hand side vector $b$ as $b=(\zeta^3,0)^\top$, it then has only one solution $(\zeta,0)^\top$ to the multilinear system. In our experiments, we consider the aforementioned two scenarios and take $\zeta=2$.
\end{example}

As shown in Example \ref{exam2}, the multilinear system with a nonnegative but not positive right-hand
side vector $b$ has an analytic nonnegative solution with zero components. We will use this example to show that our NPA can successfully find a nonnegative but not positive solution, while HM, QCA and NLSQ may obtain a fully positive solution. Since this problem is extremely simple, we solve the original system without scaling technique and use the similar stopping criterion defined in \eqref{ne5} with ${\rm Tol}=10^{-10}$ for all methods.
For the both scenarios, i.e., $b=(0,2^3)^\top$ and $b=(2^3,0)^\top$, we test three different initial points $x^0$ for the three methods, respectively. All results are listed in Tables \ref{tab3} and \ref{tab4},
where $`-$'  means that a method fails to find a solution because either the subproblem approaches to a singular linear system subproblem or the number of iterations exceeds the preset maximum iteration $2000$, `\textsf{iter}' denotes the number of iterations,
`\textsf{time}' represents the computing time in seconds and `\textsf{solution}' corresponds to an approximate solution obtained by a method.

\begin{table}[!htbp]
\caption{Numerical results for Example \ref{exam2}: (i). $b=(0,2^3)^\top$.}\label{tab3}
{
\def\temptablewidth{1\textwidth}
\begin{tabular*}{\temptablewidth}{@{\extracolsep{\fill}}llllll}\toprule
& $x^0=(0,20)^\top$&& $x^0=(20,0.001)^\top$ && $x^0=(20,20)^\top$\\
\cline{2-2} \cline{4-4}\cline{6-6}
Alg. & \textsf{iter} / \textsf{time} / \textsf{solution} && \textsf{iter} / \textsf{time} / \textsf{solution} && \textsf{iter} / \textsf{time} / \textsf{solution} \\\midrule
HM &   -- / --  / -- & &  27 / 0.41  / $(4.0,2.0)^\top$ & &  13 / 0.16  / $(4.0,2.0)^\top$ \\
QCA&   -- / --  / -- &&  51 / 0.47  / $(4.0,2.0)^\top$ &&  42 / 0.38  / $(4.0,2.0)^\top$ \\
NPA&  40 / 0.64  / $(0.0,2.0)^\top$ &&   1 / 0.31  / $(0.0,2.0)^\top$ &&   1 / 0.27  / $(0.0,2.0)^\top$ \\
NLSQ&  10 / 0.56  / $(0.0,2.0)^\top$ &&   9 / 0.11  / $(4.0,2.0)^\top$ &&  17 / 0.27  / $(0.0,2.0)^\top$ \\
 \bottomrule
\end{tabular*}}
\end{table}

\begin{table}[!htbp]
\caption{Numerical results for Example \ref{exam2}: (ii). $b=(2^3,0)^\top$.}\label{tab4}
{\scriptsize
\def\temptablewidth{1\textwidth}
\begin{tabular*}{\temptablewidth}{@{\extracolsep{\fill}}llllll}\toprule
& $x^0=(0.001,20)^\top$&& $x^0=(20,0)^\top$ && $x^0=(20,20)^\top$\\
\cline{2-2} \cline{4-4}\cline{6-6}
Alg. & \textsf{iter} / \textsf{time} / \textsf{solution} && \textsf{iter} / \textsf{time} / \textsf{solution} && \textsf{iter} / \textsf{time} / \textsf{solution} \\\midrule
HM &  13 / 0.36  / $(2.000,0.000)^\top$ & &   -- / --  / -- & &  13 / 0.23  / $(2.0,0.0)^\top$ \\
QCA&  -- / --  / -- &&   -- / --  / -- && -- / --  / -- \\
NPA&  30 / 0.80  / $(2.000,0.000)^\top$ &&  40 / 0.64  / $(2.000,0.000)^\top$ &&  30 / 0.77  / $(2.0,0.0)^\top$ \\
NLSQ&  22 / 0.52  / $(2.002,0.003)^\top$ &&  19 / 0.17  / $(2.001,0.002)^\top$ &&  -- / --  / -- \\
 \bottomrule
\end{tabular*}}
\end{table}

Notice that both HM \cite{Han17} and QCA \cite{HLQZ18} took their starting points as $x^0=b^{1/(m-1)}$ in their numerical experiments. In fact, such a positive initial point can ensure that their subproblems are nonsingular in the iterative procedure. However, if we take a nonnegative but not positive initial point $x^0$, it seems from Tables \ref{tab3} and \ref{tab4} that both HM and QCA are no longer valid, while NPA produces a nonnegative solution with zero components. If we take a fully positive initial point, it can be seen from Table \ref{tab3} that HM, QCA and NLSQ may find a fully positive solution when the multilinear system has multiple nonnegative solutions including at least one fully positive solution. Interestingly, we can observe from Table \ref{tab4} that HM and NPA can successfully obtain a nonnegative solution with zeros when taking a fully positive starting point. Thus, we guess empirically that HM is also available to find the nonnegative solution by setting an appropriate positive starting point when the multilinear system has a unique nonnegative solution. However, HM may fail to find a desired nonnegative solution if the system has one more fully positive solution. Promisingly, the proposed NPA works well with different initial points for Example \ref{exam2}. Comparatively speaking, the proposed NPA seems more robust on finding nonnegative solutions when starting with different initial points.

Below, we consider some higher order sparse nonsingular ${\mathcal M}$-tensors, which are generated randomly as follows. We first randomly generate a sparse nonnegative tensor ${\mathcal B}\in{\mathbb T}_{m,n}$, whose $80\%$ components are zeros and others are uniformly distributed in $(0,1)$. Then, by setting $\omega=0.1$ in
 \begin{equation*}\label{ne1}
s=(1+\omega)\cdot \max_{i\in[n]}\left(\sum^n_{i_2,\cdots,i_m=1}b_{ii_2\cdots i_m}\right)
\end{equation*}
and letting ${\cal A}:=s{\cal I}-{\cal B}$, we can see that ${\cal A}$ is a nonsingular ${\cal M}$-tensor since
\begin{equation*}\label{ne2}
\rho({\cal B})\leq \max_{i\in[n]}\left(\sum^n_{i_2,\cdots,i_m=1}b_{ii_2\cdots i_m}\right),
\end{equation*}
and $s>\rho({\cal B})$. For the vector $b$, we first generate a sparse vector $x^*=\verb"sprand"(n,1,0.4)$ by the {\sc Matlab} script `\verb"sprand"', where $0.4$ controls the number of zero components, and then let $b={\mathcal A}(x^*)^{m-1}\in{\mathbb R}^n_+$. Therefore, we can always ensure that the resulting multilinear system has at least one nonnegative but not positive solution.

In this test, we use $x^0=(1,1,\cdots,1)^\top$ and ${\rm Tol}=10^{-10}$ to be the starting points and tolerance for all methods, respectively. Here, we only plot the convergence curve of NPA on this example in Fig. \ref{fig1} to further support our conjecture (i.e., linear convergence rate behaviour of NPA). Moreover, we compare the approximate solutions obtained by all methods with the known true solution of the multilinear system in Fig. \ref{fig2}.

\begin{figure}[!htbp]
\includegraphics[width=0.95\textwidth]{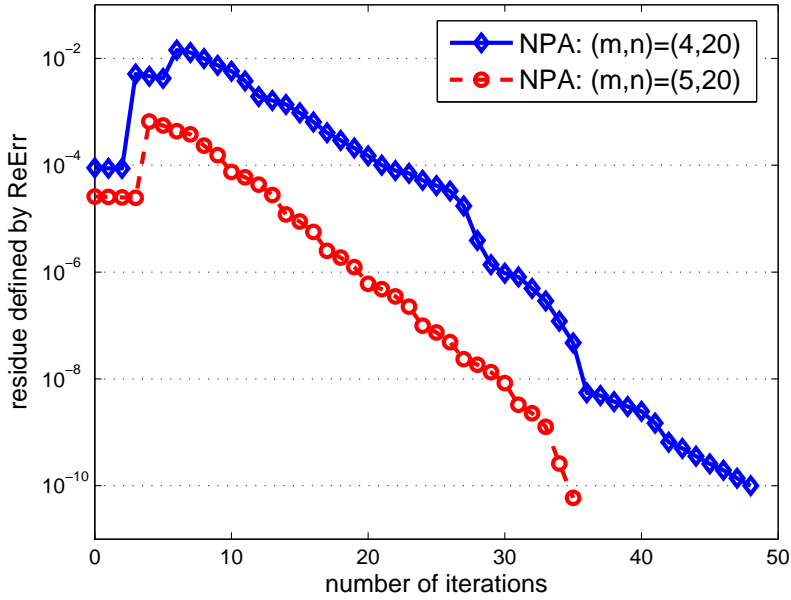}
\caption{Evolutions of the residue defined by \textsf{ReErr} in \eqref{ne5} with respect to the number of iterations. The convergence curve implies that NPA may be linearly convergent.}
\label{fig1}
\end{figure}

\begin{figure}[!htbp]
\includegraphics[width=0.49\textwidth]{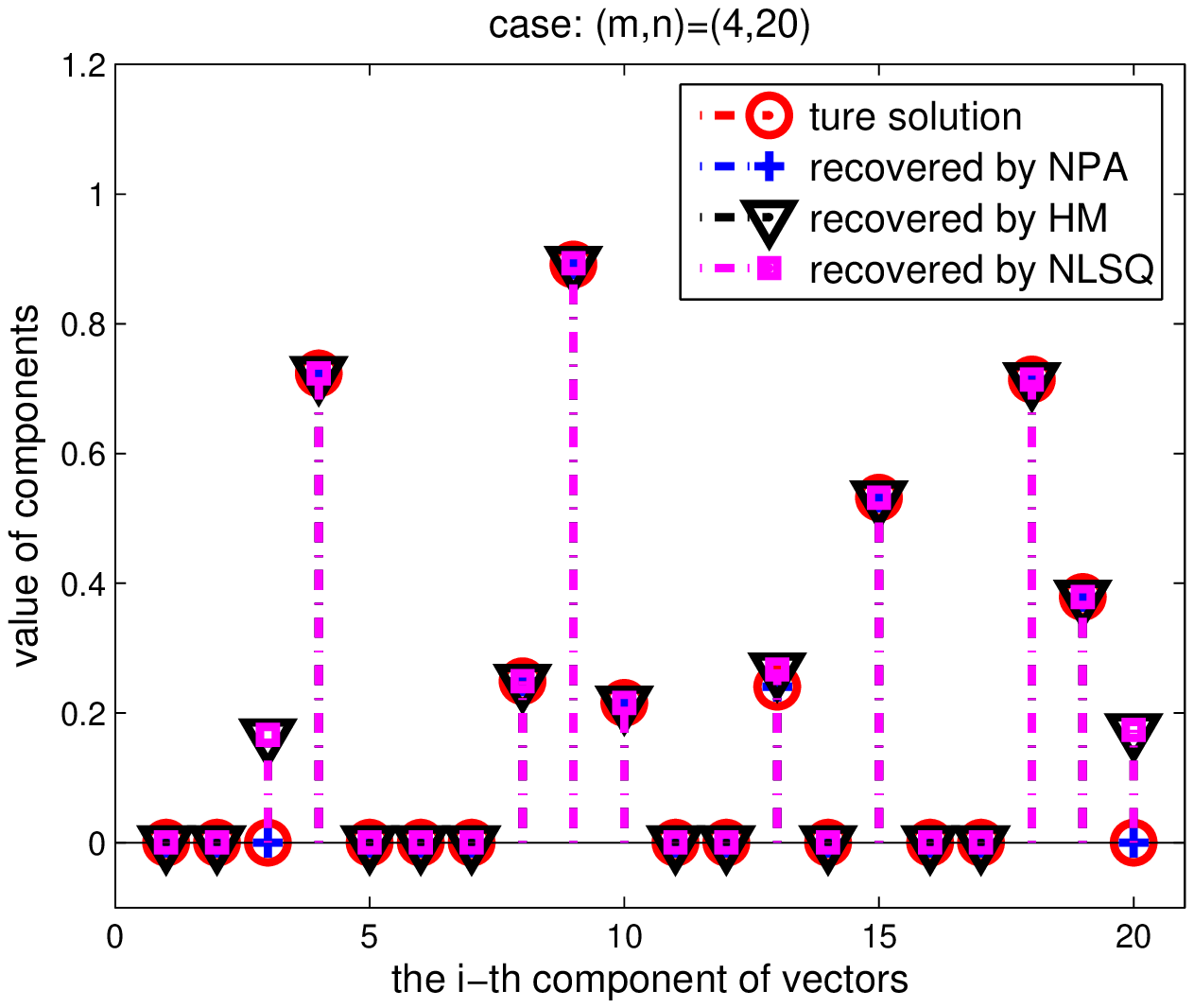}
\includegraphics[width=0.49\textwidth]{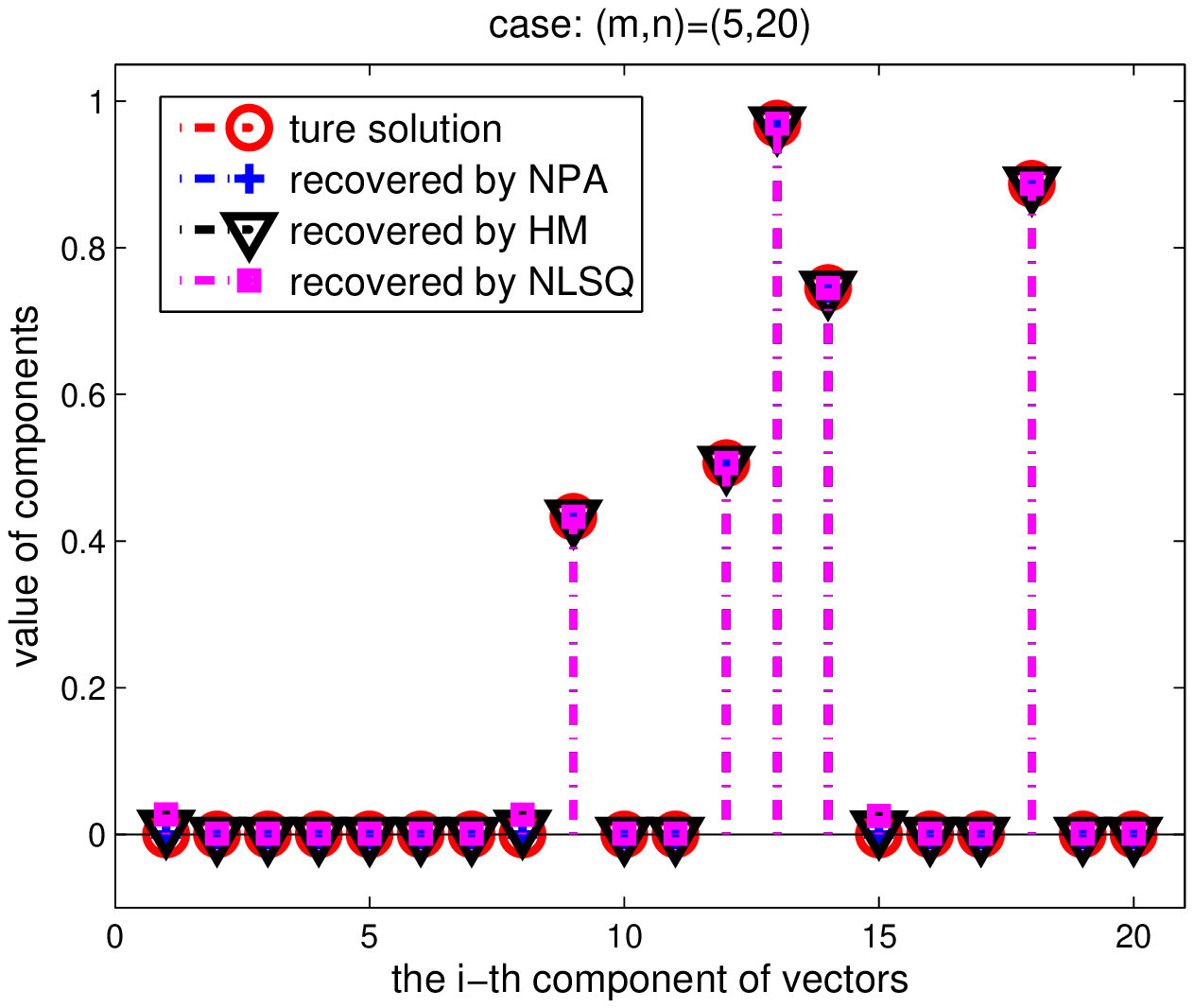}
\caption{Comparison on the solutions recovered by the methods.}
\label{fig2}
\end{figure}
It can be easily seen from Figs. \ref{fig1} and \ref{fig2} that NPA  can return a sparse nonnegative solution to the multilinear system. The convergence curves in Fig. \ref{fig1} show that NPA  seems to be linearly convergent. Moreover, Fig. \ref{fig2} show that NPA can perfectly find a nonnegative sparse solution to the multilinear system under test. For the case $(m,n)=(4,20)$, we can see that HM and NLSQ return a relatively lower quality solution. Hence, the promising nonnegativity preserving property of NPA might be helpful to algorithmic design for sparse nonnegative tensor equations studied in \cite{LN15,LQX15}, which is also one of our future research topics.

Finally, we further test a $3$-rd order higher dimension case with nonnegative $b$, where we generate sparse data in a similar way used in Figs. \ref{fig1} and \ref{fig2}, i.e., tensor ${\mathcal B}\in{\mathbb T}_{m,n}$ is sparse with $60\%$ zeros and the nonnegative $b$ is generated by {\sc Matlab} script `$b$=\verb"sprand"$(n,1,0.4)$'. As shown in Table \ref{tab4}, QCA is not valid  for the case where $b$ is nonnegative with zeros.  Therefore, we only compare HM, NPA, and NLSQ. Since the randomness of ${\mathcal B}$ and $b$, we randomly generate $100$ groups of the data and report the numerical performance of HM, NPA, and NLSQ in Fig. \ref{fig3}. In our experiments, we take $x^0=(10,\cdots,10)^\top$ and ${\rm Tol}=10^{-8}$ and the maximum iteration being $500$ for all methods. For the rate reported in Fig. \ref{fig3}, it can be regarded that the problem is successfully solved if the residue defined by \eqref{ne5} is less than $10^{-5}$; otherwise, it can be regarded as failure. Moreover, we regard a component of an approximate solution as zero if the value of the component is less than $10^{-5}$; otherwise, it is a nonzero (positive) component.

\begin{figure}[!htbp]
\includegraphics[width=0.49\textwidth]{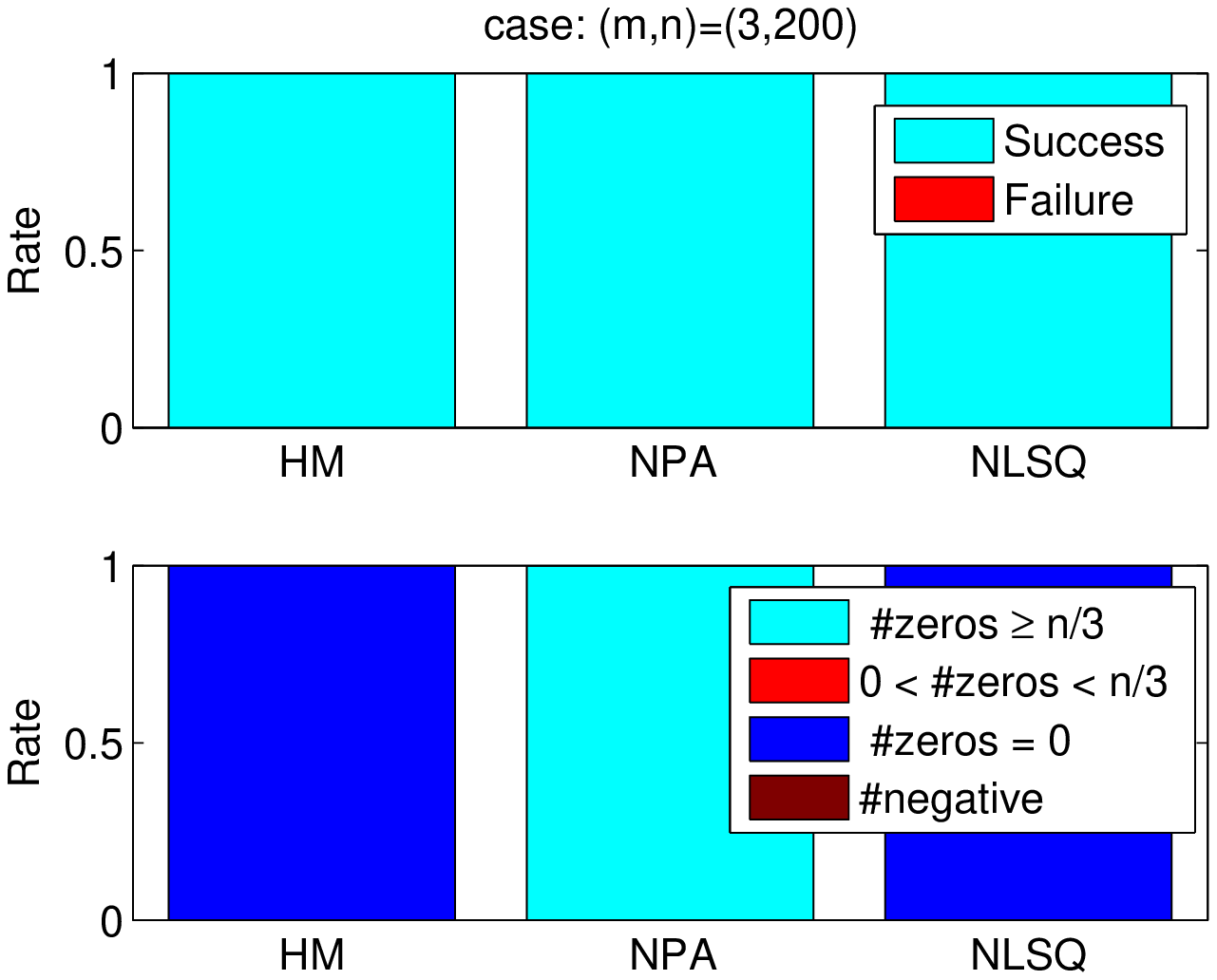}
\includegraphics[width=0.49\textwidth]{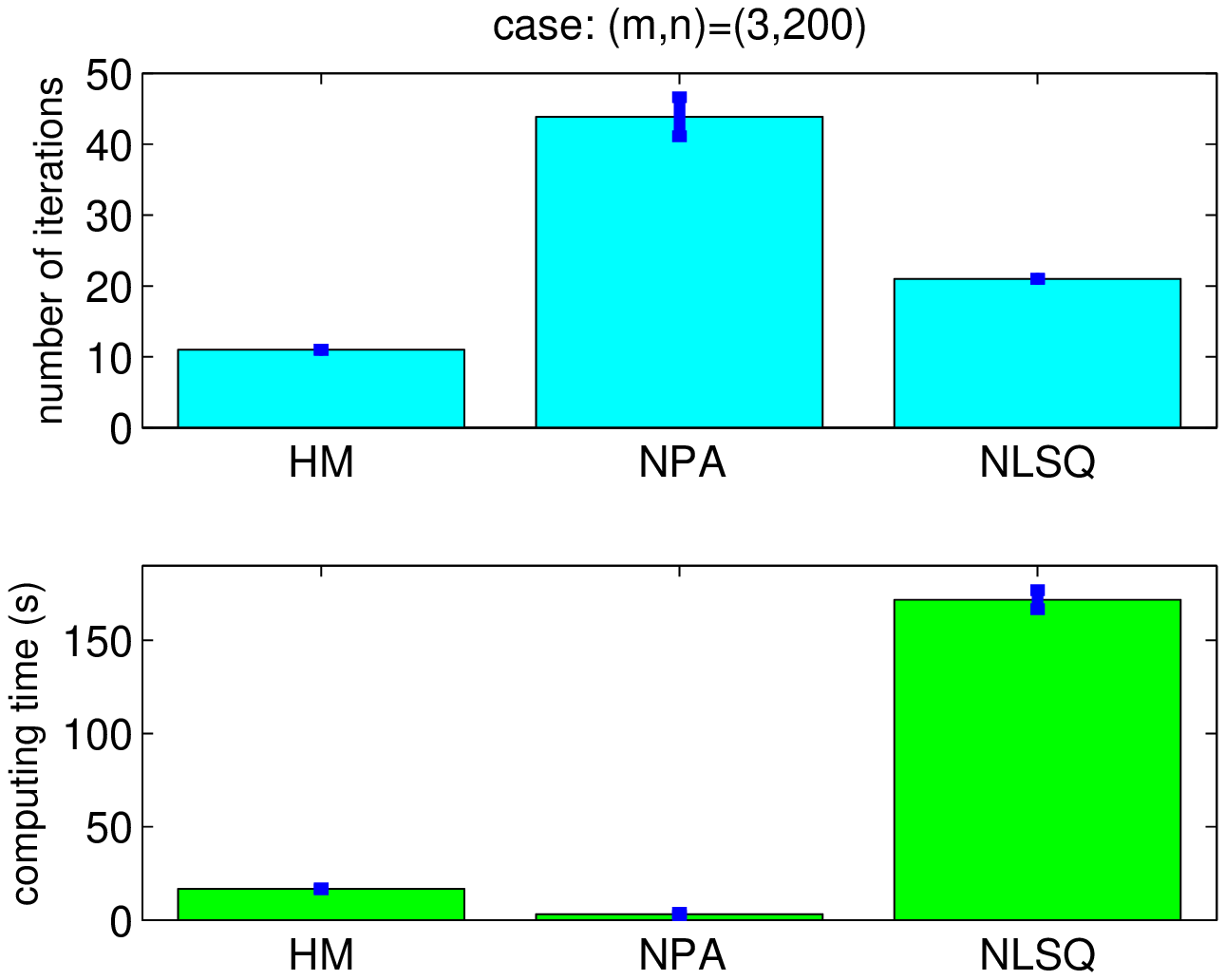}
\caption{Comparison for general cases of \eqref{mle} ($(m,n)=(3,200)$) with nonnegative $b$. (i) The top left subplot corresponds to success rate and failure rate, and the bottom left subplot corresponds to the rate of different types (nonnegative sparse [\#zeros$\geq \frac{n}{3}$], nonnegative with zeros [$0< $\#zeros$<\frac{n}{3}$], fully positive [\#zeros$=0$], and negative) of solutions; (ii) The right two subplots are the standard error bar on iterations and computing time.}
\label{fig3}
\end{figure}

It is clear from Fig. \ref{fig3} that HM, NPA, and NLSQ can successfully find a solution to \eqref{mle} with a nonnegative vector $b$. However, both HM and NLSQ only obtain positive solutions to the problem. In this case, if ones were concerned about sparse solutions to \eqref{mle},  results here show that NPA seems the most reliable solver. Moreover, the right two standard error bars in Fig. \ref{fig3} show that NPA takes less computing time to find a (sparse) solution than both HM and NLSQ, even though NPA requires more iterations. From this point, we think that the proposed NPA is efficient for the problem under consideration.

According to the results reported in this section, we can draw the conclusion that, compared to HM \cite{Han17}, QCA \cite{HLQZ18}, and NLSQ (`\verb"lsqnonlin"'), the proposed NPA (Algorithm \ref{alg31}) has its own advantages, i.e., it can be applied to a wider range of cases. In particular, when the multilinear system has multiple nonnegative solutions, and if our purpose is to get a solution as sparse as possible, the proposed algorithm may be a better candidate solver to achieve this goal.

\section{Conclusion}\label{s5}
In this paper, we mainly studied the multilinear system in the form of \eqref{mle}. We showed that the multilinear system, whose coefficient tensor is a nonsingular ${\mathcal M}$-tensor and right-hand side vector is nonnegative, always has a nonnegative solution, but the solution may not be unique. Aiming at this case, we proposed a Newton-type algorithm that can perfectly preserve the nonnegativity of the iterative sequence. Moreover, we show that a nonnegative decreasing sequence generated by our proposed algorithm converges to a nonnegative solution of the system under consideration. By numerical experiments, we stated that our method is efficient and it has advantages over some existing algorithms: when the right-hand side is nonnegative but not positive, our proposed algorithm can still output a nonnegative solution of the system, while the others may not produce a nonnegative solution. In the future, we will try to analyze the convergence rate of the proposed algorithm and apply it to real-life sparse problems.

\medskip
\begin{acknowledgements}
The authors are grateful to the editors, the two anonymous referees and Professor Donghui Li for their valuable comments which led to great improvements of the paper. H. He and C. Ling were supported in part by National Natural Science Foundation of China (Nos. 11771113 and 11571087) and Natural Science Foundation of Zhejiang Province (Nos. LY19A010019 and LD19A010002).
\end{acknowledgements}


\end{document}